\documentclass{article}

\usepackage{amsfonts}
\usepackage{amsmath}
\usepackage{amsthm}
\usepackage{mathrsfs}

\DeclareMathOperator{\row}{Row}
\DeclareMathOperator{\rowmax}{row-max}
\DeclareMathOperator{\rowargmax}{row-argmax}

\DeclareMathOperator{\colmin}{col-min}
\DeclareMathOperator{\colargmin}{col-argmin}
\DeclareMathOperator{\argmin}{argmin}
\DeclareMathOperator{\argmax}{argmax}
\DeclareMathOperator{\col}{Col}

\DeclareMathOperator{\floor}{Floor}
\DeclareMathOperator{\ceil}{Ceil}

\newcounter{mscount}

\newcommand{\str}[1]{\mathbb{#1}}

\newtheorem{theorem}{Theorem}
\newtheorem{definition}[theorem]{Definition}
\newtheorem{example}[theorem]{Example}

\newtheorem{proposition}[theorem]{Proposition}

\newcommand{\val}{\mathscr{V}}
\newcommand{\tally}[1]{{\Sigma}#1}
\newcommand{\intrp}[1]{[\![#1]\!]}
\newcommand{\phimp}{\phi_{\mathrm{MP}}}

\title{Linear Programming Tools for Analyzing Strategic Games of Independence-Friendly Logic and Applications}
\author{Merlijn Sevenster}

\begin{document}

\maketitle

\abstract{In recent work, semantic games of independence-friendly logic were studied in strategic form in terms of (mixed strategy) Nash equilibria. The class of strategic games of independence-friendly logic is contained in the class of win-loss, zero-sum two-player games. In this note we draw on the theory of linear programming to develop tools to analyze the value of such games. We give two applications of these tools to independence-friendly logic under the so-called equilibrium semantics.}

\section{Introduction}
\label{Sec:Introduction}

At the heart of game-theoretic semantics \cite{Hintikka:73,Hintikka:96} lies the understanding that meaning emerges as the result of the interaction between rational agents who act in their own interest according to a set of rules. The meaning that arises thus is attached to the linguistic expression that constitutes this set of rules. The concept of game was used to flesh out this understanding, a concept that was also applied by Wittgenstein to the philosophy of (natural) language \cite{Wittgenstein:58}. A sample game, well known in the context of independence-friendly logic, is defined by the following set of rules parameterized by a function $f$. First an opponent chooses an $x$ and $\varepsilon>0$ on the reals. Then we choose a $\delta>0$ independent of $x$. The interaction terminates after the opponent has chosen a $y$. We ``win'' if the series of objects satisfies the condition
\begin{equation}
\label{Eq:ContinuousFunction}
\textrm{$|x-y| < \delta$ implies $|f(x) - f(y)| < \varepsilon$},
\end{equation}
otherwise the opponent ``wins''.

The linguistic expression of this rule set, in the syntax of independence-friendly (IF) logic, is
\[
	\forall x \forall \varepsilon (\exists \delta/x)\forall y \psi(x,\varepsilon,\delta,y),
\]
where $\psi(x,\varepsilon,\delta,y)$ is a formalization of (\ref{Eq:ContinuousFunction}). In game-theoretic semantics, the meaning of the latter expression, or any IF sentence for that matter, is defined as the conditions under which we win the game. We will make this more precise in due course.

For now it is important to emphasize that game-theoretic semantics puts rule-governed interaction at the center of attention, and that meaning is derived from it. This view must be contrasted to the view according to which the meaning of a logical expression is determined by conditions under which it is ``true'' in a (formalized) state of affairs. It seems that Tarski semantics for first-order logic is a formal counterpart of this view.

Interaction between self-interested agents is studied in game theory. In this area, a \emph{(pure) strategy} for a player completely specifies how to move in each choice point for that player. If we are given one strategy of each player in a game (in the formal, game-theoretic sense of the word), then we can traverse the sequence of choice points that arise if we follow the moving player's strategy. The respective players' \emph{payoffs} are distributed among the players once the terminal node in this sequence is reached. A win-loss game is a game in which the players can either win (i.e., receive payoff 1) or lose (i.e., receive 0). In the context of a win-loss game, a strategy is \emph{winning} if it results in a win for its owner against each strategy of its opponent.

Game-theoretic semantics for independence-friendly logic was developed (first in spirit \cite{HintikkaSandu:89}, then in formalism \cite{Dechesne:PhD:05}) in the framework of extensive games. This framework considers a game as a \emph{game tree} in which each node corresponds to a choice point for a player or a terminal node (i.e., a node in which payoff is returned to the players). Folklore has it that semantic games of IF logic are played between Eloise and Abelard. The game tree that formalizes the interactions between Eloise and Abelard constituting the meaning of an IF sentence $\phi$ in the context of a suitable structure $\str{M}$ is called an \emph{extensive game of imperfect information}, denoted $G(\str{M},\phi)$.

The meaning of $\phi$ can be seen to emerge from interaction in $G(\str{M},\phi)$, by inspecting that the condition
	\begin{equation} \label{Eq:TrueTarski}
		\textrm{$\phi$ is true on $\str{M}$ (written $\str{M} \models^+ \phi$)}
	\end{equation}
coincides with the condition
	\begin{equation} \label{Eq:TrueGTS}
		\textrm{Eloise has a winning strategy in the game $G(\str{M},\phi)$},
	\end{equation}
and that
	\begin{equation} \label{Eq:FalseTarski}
		\textrm{$\phi$ is false on $\str{M}$ (written $\str{M} \models^- \phi$)}
	\end{equation}
coincides with
	\begin{equation} \label{Eq:FalseGTS}
		\textrm{Abelard has a winning strategy in the game $G(\str{M},\phi)$}.
	\end{equation}

From a logical point of view, game-theoretic semantics has several advantages. Its tree-based view nicely reflects the dependence between nested quantifiers. Conditions (\ref{Eq:TrueTarski}) to (\ref{Eq:FalseGTS}) show us how one particular type of interaction coincides with the meaning of IF sentences, at least in terms of their truth and falsity conditions. 

The equivalences between Conditions (\ref{Eq:TrueTarski}) to (\ref{Eq:FalseGTS}) show us how known grounds --- i.e., the Tarskian notions of truth and falsity --- can be covered by game-theoretical semantics. They also give us a lead for exploring uncharted territory. Namely, we see that the above conditions are based on one mode of interaction only, that is, one player having a winning strategy. Thus a whole research agenda unfolds itself in front of us: analyzing the interrelations between game-theoretic interactions on the one hand and the meanings that arise from them on the other hand. The Matching Pennies sentence $\phimp$ serves to illustrate the type of questions that motivate this agenda:
\[
	\forall x(\exists y/x) x = y
\]
On structures $\str{M}$ with more than one element, neither player has a winning strategy in the game $G(\str{M},\phimp)$. Such games are said to be \emph{undetermined}. The games $G(\str{M},\phimp)$ are not covered by the Conditions (\ref{Eq:TrueGTS}) and (\ref{Eq:FalseGTS}), as if no meaning can be seen to emerge from them. 

This observation can be made more precise. Every IF sentence $\phi$ partitions the class of suitable structures in three:
	\[
		\big( 
			\intrp{\phi}^-
			, 
			\intrp{\phi}^{\not\pm}
			, 
			\intrp{\phi}^+
		\big),
	\]
where $\intrp{\phi}^+$ denotes the set of structures on which Eloise has a winning strategy, $\intrp{\phi}^-$ denotes the set of structures on which Abelard has a winning strategy, and $\intrp{\phi}^{\not\pm}$ contains the other structures, i.e., the structures on which neither player has a winning strategy. In the case of the Matching Pennies sentence, $\intrp{\phimp}^+$ is the set of structures with one element; $\intrp{\phimp}^-$ is empty; and $\intrp{\phimp}^{\not\pm}$ is the set of structures with more than one element. The observation that game-theoretic semantics does not cover the undetermined games of $\phimp$, i.e. the games on $\intrp{\phimp}^{\not\pm}$, touches on the following question: How can we give a \emph{direct} definition of $\intrp{\phi}^{\not\pm}$? --- understanding that its present definition, stated in terms of the absence of a winning strategy for either player, is indirect.

We can also study the class of semantic games as a game-theoretic entity in its own right and ignore the fact that each game in this class is constituted by an IF sentence and a structure. From such a point of view, it is only natural\footnote{The author is grateful to Allen L.~Mann for suggesting this point of view.} to generalize from pure strategies to \emph{mixed strategies}, those being the dominant species of strategies in game theory. Mixed strategies are studied more naturally in the framework of \emph{strategic games}, which ignore the games' sequential turn-taking dynamics. Finally, instead of studying winning (pure) strategies, we can now shift our attention to equilibrium mixed strategies, that is, mixed strategies that cannot be improved upon by any of the players.

In a recent publication \cite{SevensterSandu:10}, rooted in an observation by Ajtai \cite{BlassEtAl:86} and anticipated in \cite{Sevenster:06,Galliani:09}, the strategic game theory of IF games was developed. A mixed strategy is a probability distribution over a set of pure strategies. If Eloise and Abelard both play mixed strategies, that is, if they pick their pure strategies at random according to their mixed strategies, the pair of pure strategies that will be played is effectively selected from the lottery determined by the product of their mixed strategies. If we associate payoff $0$ with Eloise losing the outcome of playing two pure strategies against each other, and $1$ with her winning, we can define the \emph{expected payoff} of Eloise as the expected utility that is returned to her in this lottery. It is not hard to see that Eloise's expected utility falls in $[0,1]$. 

Informally, an equilibrium is a state of the game in which the players' powers to influence the outcome are in balance, or, somewhat more formally, it is a pair of mixed strategies in which neither player can benefit from unilateral deviation. The \emph{value} of a strategic game between Eloise and Abelard is defined as Eloise's expected utility of an equilibrium.

In \cite{SevensterSandu:10}, the strategic IF game $\Gamma(\str{M},\phi)$ was defined as the strategic counterpart of the extensive game $G(\str{M},\phi)$. It was further postulated that the \emph{value} of $\phi$ on $\str{M}$ is the value of $\Gamma(\str{M},\phi)$, that is, Eloise's \emph{expected utility} in $\Gamma(\str{M},\phi)$. For instance, as is easily proven (see also Example \ref{Ex:Prop:Equivalence:1} below), the Matching Pennies sentence has value $1/n$ on structures of size $n$. The notation $\str{M} \models_\varepsilon \phi$ was introduced to indicate that $\phi$ has value $\varepsilon$ on $\str{M}$. We will introduce the framework of \emph{equilibrium semantics} and key results more rigorously in the next section, including the results by which every finite strategic IF game has one unique value. 

The strategic view disregards the sequential turn-taking of the games trees, which are so nicely reflected the quantifier alternation of IF sentences. In return we get a formalism in which the notion of strategy is atomic. This somehow matches the way in which interaction is primitive in the philosophy behind game-theoretic semantics. Furthermore, it is to be understood that the strategic view on semantic games is a generalization of the extensive view, in the sense that conditions can be found in terms of equilibrium mixed strategies that are equivalent to Conditions \ref{Eq:TrueTarski} and \ref{Eq:FalseGTS}. Indeed, one of the first results about equilibrium semantics, reiterated in the next section, has it that
\begin{align}
	\intrp{\phi}^- & = \intrp{\phi}^0 \\
	\intrp{\phi}^+ & = \intrp{\phi}^1,
\end{align}
where $\intrp{\phi}^\varepsilon$ denotes the class of structures on which $\phi$ has value $\varepsilon$. This result shows that equilibrium semantics is a \emph{conservative extension} of traditional game-theoretic semantics. It also shows that $\intrp{\phi}^{\not\pm}$, which was defined indirectly in game-theoretic semantics, can be defined directly in equilibrium semantics:
	\begin{align}
		\intrp{\phi}^{\not\pm} & = \bigcup_{\varepsilon \in (0,1)} \intrp{\phi}^\varepsilon.
	\end{align}

It is yet to be seen what type of meaning is constituted by the interaction studied in equilibrium semantics. At present, no coherent semantic interpretation has been given of very the notion of value. In an attempt to get a handle on the problem of interpreting $\models_\varepsilon$ consider the partitioning
	\[
		\big( \intrp{\phi}^\varepsilon \big)_{\varepsilon \in [0,1]}
	\]
for any IF sentence $\phi$. For instance, $\intrp{\phimp}^{1/n}$ contains the structures of size $n$. With each class $\intrp{\phi}^\varepsilon$ we can seek a logical expression $\psi^\varepsilon$ that defines it, in the sense that $\str{M} \in \intrp{\phi}^\varepsilon$ if, and only if, $\psi^\varepsilon$ is true on $\str{M}$. From a model-theoretic point of view we are interested in the logical languages in which such $\psi$ can be defined, whereas, from a more philosophical viewpoint, we are interested to learn the interrelations between the sentences in
	\[
		\big( \psi^\varepsilon \big)_{\varepsilon \in [0,1]}.
	\]
For instance, how does $\psi^\varepsilon$ relate to $\psi^1$, which expresses $\phi$'s truth conditions, and $\psi^0$, which expresses its falsity conditions?

One of the obstacles we are facing in this respect is the informal way of thinking about game-theoretical semantics. For instance, we grew used to thinking of Eloise ``wanting'' to prove that the sentence $\phi$ is true, and Abelard ``wanting'' to prove that it is false. In the case of the Matching Pennies sentence, this would mean that Eloise wants to establish that the structure has one element and that Abelard wants to establish the logical contradiction, whatever that may mean.

In equilibrium semantics, it is unclear what semantic relation Eloise and Abelard want to establish between $\phi$ and $\str{M}$. For all we know, Eloise and Abelard want to maximize their payoff in $\Gamma(\str{M},\phi)$, but what does that tell us about the relation between $\phi$ and $\str{M}$?

Another obstacle for understanding $\models_\varepsilon$ is the fact that we lack tools to analyze strategic IF games. Establishing the series $( \psi^\varepsilon)_{\varepsilon \in [0,1]}$ of seemingly simple $\phi$ may take several pages of text, especially if $\phi$ is interpreted on arbitrary graph-like structures. To grasp this point it is instructive to realize that the problem of determining the value of an arbitrary win/loss game with values 0 and 1 reduces to the problem of determining the value of $\forall x (\exists y/x) R(x,y)$, in the sense that if we have an algorithm to solve the latter, we can tweak it to solve the former. 

This computational concern touches on the worst-case computational complexity of determining the value of an arbitrary win/loss game. It is known that this problem can be defined as a linear programming problem, for which efficient (polynomial time) algorithms have been proposed. Unfortunately, these algorithm are fairly intricate, which as yet renders them quite useless, in their current forms, for establishing the value of strategic IF games.

In this paper, we shall exploit the linear programming view on strategic IF games to develop a set of tools for determining and approximating their value. It is important to realize that the tools developed in this way are weaker than the efficient algorithms that were proposed earlier to solve arbitrary linear programming problem. If our results have any merits, it may be in the fact that they help us to more easily determine the value of certain strategic IF games, or that they inspire the construction of more powerful tools.

In the next section, we will review definitions and elementary results of equilibriums semantics. In Section \ref{Sec:Games} we present some tools to analyze win-loss, zero-sum, two-player strategic game. In Section \ref{Sec:Applications} we apply these tools to analyze the values of two IF sentences that pertain to the birthday problem and hashing.

\section{Preliminaries}
\label{Sec:Preliminaries}

An extensive game $G$ describes all positions of the game and how it proceeds from one to the other. In an extensive game with players $P$, each player $p \in P$ has a set of \emph{(pure) strategies} $S_p$. A pure strategy is essentially a rule book that prescribes how its owner moves in every position of the game (i.e., \emph{history}) in which it is his/her turn. A definition of extensive games can be found in \cite{Dechesne:PhD:05,MannEtAl:11}.

A \emph{strategy profile (for the players $P$)} $\bar{\sigma}$ is a function that selects an appropriate strategy for each player in $P$. If $P = \{p_0,\ldots,p_{n-1}\}$ we shall also write $\bar{\sigma}$ as the sequence $(\sigma_{p_0},\ldots,\sigma_{p_{n-1}})$. An extensive game $G$ has a utility function $u_p$ for each player $p$ that assigns a real value to each of the game's strategy profiles.

We shall be interested in two-player games, so that our strategy profiles contain two strategies. We shall use the symbols $\exists$ and $\forall$ to mark the game's contestants, Eloise and Abelard: $P=\{\exists,\forall\}$. Moreover we shall focus on \emph{win-loss} and \emph{zero-sum} games, that is, the utility functions $u_\exists$ and $u_\forall$ will be functions with range $\{0,1\}$ such that for each strategy profile $(\sigma,\tau)$, $u_\exists(\sigma,\tau) + u_\forall(\sigma,\tau)=1$. Since in this type of games, $u_\forall$ is uniquely determined by $u_\exists$, we shall simply write $u$ for $u_\exists$ and mostly ignore $u_\forall$.

A pure strategy $\sigma \in S_p$ is \emph{winning} in a win-loss, zero-sum game if $u_p(\sigma,\tau) =1$ for each strategy $\tau$ of $p$'s opponent $\bar{p}$.

Independence-friendly logic is the extension of first-order logic whose quantifiers $(Qx/X)$ are furnished with sets of variables $X$ indicating that the choice of quantifier $Qx$ be made independent from the variables in $X$. In this paper we shall only use the syntax of IF logic when we apply our game-theoretic results to express certain properties. We refer the reader to \cite{MannEtAl:11} for a comprehensive introduction to the field of IF logic, which also introduces more gently the basic notions of equilibrium semantics.

Sentences of IF logic are evaluated on \emph{structures} 
\[
	\str{M}=\big(M,R_0^\str{M},R_1^\str{M},\ldots,f_0^\str{M},f_1^\str{M},\ldots\big),
\] 
where $M$ is the \emph{universe} of $\str{M}$, $R_i^\str{M}$ is the \emph{interpretation} of \emph{relation symbol} $R_i$ and $f_i^\str{M}$ is the interpretation of \emph{function symbol} $f_i$, as usual.

The semantic game of an IF sentence $\phi$ on a structure $\str{M}$ gives rise to the \emph{(extensive) IF game} 
\(
 	G(\str{M},\phi)
\),
which is a two-player, win-loss and zero-sum game. It is also a game of imperfect information if $\phi$ has quantifiers $(Qx/X)$ in which $X$ is nonempty.

The framework of strategic game theory gives another way of looking at games. Suppose that $G$ is the extensive formalization of a game. Then, the strategic form of the same game would be 
\[
	\Gamma = \big( (S_p)_{p \in P}, (u_p)_{p \in P} \big),
\]
where $P$ is the set of players as before, $S_p$ is the set of $p$'s pure strategies in $G$, and $u_{p}$ is player $p$'s utility function in $G$. The strategic game $\Gamma$ is two-player/win-loss/zero-sum, whenever $G$ is. We shall write $\Gamma(\str{M},\phi)$ for the \emph{strategic IF game} that is the strategic counterpart of the extensive IF game $G(\str{M},\phi)$.

A \emph{mixed strategy} $\mu_p$ of player $p$ in $\Gamma$ is a probability distribution over $S_p$, that is, $\mu_p$ is a function for which for every $\sigma \in S_p$, $0 \leq \mu_p(\sigma) \leq 1$, and $\sum_{\sigma \in S_p} \mu_p(\sigma) = 1$. The mixed strategy $\mu_p$ is \emph{uniform} if it assigns equal probability to each pure strategy in $S_p$. We say that $\mu_p$ is \emph{uniform in $T$}, for any $T \subseteq S_p$, if the domain of $\mu_p$ is $T$ and if it assigns equal probability to each pure strategy in $T$.

We extend the notion of strategy profile to mixed strategies; whence a strategy profile may also refer to a sequence $(\mu_p)_{p \in P}$ of mixed strategies. A strategy profile of mixed strategies defines a \emph{lottery} over the set of outcomes of the game, that is, strategy profile $(\sigma_p)_{p\in P}$ is drawn with likelihood
\[
	\prod_{p \in P} \mu_p(\sigma_p).
\] 
The \emph{expected utility} for player $p$ is given by $p$'s expected utility in the lottery. For a strategy profile of mixed strategies $(\mu_\exists,\mu_\forall)$, the expected utility is defined as as
	\[
		U_p(\mu_\exists,\mu_\forall) = \sum_{\sigma \in S_\exists}\sum_{\tau \in S_\forall} \mu_\exists(\sigma) \mu_\forall(\tau) u_p(\sigma,\tau).
	\]

If $\Gamma$ is a zero-sum and win-loss game between $\exists$ and $\forall$, then we have that $U_\exists(\mu,\nu) + U_\forall(\mu,\nu) = 1$. In this case, for the same reason as before, we shall write $U$ for $U_\exists$ and forget about $U_\forall$.

The theory of mixed strategy equilibrium predicts that Eloise and Abelard will settle on a pair of mixed strategies in which neither player benefits from unilateral deviation, that is, from choosing another mixed strategy.

\begin{definition}
Let $\Gamma$ be a two-player strategic game. The strategy profile $(\mu_\exists,\mu_\forall)$ is an \emph{equilibrium (in mixed strategies) in $\Gamma$} if for each player $p \in \{\exists,\forall\}$,
\[
	U_p(\mu_p,\mu_{\bar{p}}) \geq U_p(\mu_p',\mu_{\bar{p}})
\]
for each mixed strategy $\mu_p'$ of $p$.

A strategy in an equilibrium is called an \emph{equilibrium strategy}.
\end{definition}

The Minimax Theorem (see Theorem \ref{Th:Minimax} below) shows that every finite, two-player, zero-sum game $\Gamma$ has an equilibrium. Nash \cite{Nash:51} later generalized this result to arbitrary finite strategic games, and this type of equilibrium has henceforth been associated with his name. Since in this work we shall only require the Minimax Theorem, we shall not use the term Nash equilibrium despite the fact that it seems to be more common in the literature on game theory.

It is not hard to see that if $\Gamma$ has multiple equilibria, they all return the same expected utility to Eloise. We call this the \emph{value} of the game, and write it as $\val(\Gamma)$. We define \emph{equilibrium semantics} as the relation $\models_\varepsilon$ for which
\[
	\str{M} \models_\varepsilon \phi \quad \textrm{iff} \quad \textrm{$\val(\Gamma)=\varepsilon$},
\]
where $\Gamma = \Gamma(\str{M},\phi)$. This relation is well defined for finite structures $\str{M}$, but not necessarily on infinite structures. Thus, in this paper, we shall only consider finite structures.

The present definition of equilibrium semantics is not compositional, that is, the value of an IF formula is not determined on the basis of the values of its subformulas. Interestingly, it was shown by Galliani and Mann \cite{GallianiMann:11} that compositionality can be restored by extending Hodges' trump semantics with probability distributions over assignments. This approach may yield other tools for analyzing the values of IF strategic games.

\section{Games}
\label{Sec:Games}

In this section we take a linear programming perspective on computing the value of two-person, zero-sum strategic games that is known from the literature \cite{Raghavan:94}. This class of games contains the strategic IF games as a subclass. Thus we can use insights obtained to construct tools for computing and approximating the value of strategic IF games.

\subsection{Linear programming}
We write $0,\ldots,m-1$ for Eloise's pure strategies and $0,\ldots,n-1$ for Abelard's in a strategic game. If Eloise plays $i$ and Abelard plays $j$, Eloise receives $u(i,j) \in \{0,1\}$. Oftentimes we shall consider the payoff function $u$ as a matrix:
	\[
		\left[
			\begin{array}{ccc}
				u(0, 0)	& \cdots & u(0, n-1)\\[2mm]
		 		\vdots & \ddots & \vdots \\[3mm]
				u(m-1,0)	& \cdots  & u(m-1, n-1) 
			\end{array}
		 \right]
	\]
In fact we shall regard such matrices $u$ as games in their own right, understanding that Eloise controls the row strategies and Abelard controls the column strategies. Accordingly we write $\val(u)$ for the value of the game corresponding to the matrix $u$. \emph{Throughout this section the word ``game'' designates any matrix $u$ with entries carrying values in the range $\{0,1\}$, unless specified otherwise.}

The \emph{security level for Eloise} in a game $u$ is defined as 
\[
	\max_\mu \min_\nu U(\mu,\nu),
\] 
where $\mu$ ranges over Eloise's mixed strategies in $u$ and $\nu$ over Abelard's. It may be instructive to take a game-theoretic view on the expression $\max_\mu \min_\nu U(\mu,\nu)$. According to this view, the security level is the value that is the result of a game between Maximizer and Minimizer. In this game, Maximizer chooses a mixed strategy $\mu$ for $\max_\mu$. Then Minimizer chooses a mixed strategy $\nu$ for $\min_\nu$ knowing $\mu$. The game ends and Maximizer receives $U(\mu,\nu)$ and Minimizer receives $1- U(\mu,\nu)$. Thus the security level corresponds to the maximal value that Maximizer can secure. Similarly, the \emph{security level for Abelard} is defined as $\min_\nu \max_\mu U(\mu,\nu)$. 

Note the informational asymmetry between the games defined by
\[
	\max_\mu \min_\nu U(\mu,\nu)
\]
and 
\[
	\min_\nu \max_\mu U(\mu,\nu).
\] 
In the former game Minimizer observes the move by Maximizer before she picks her mixed strategy, whereas in the latter game Maximizer has the informational advantage. Below we shall associate Maximizer with Eloise and Minimizer with Abelard.

Von Neumann's Minimax Theorem \cite{vonNeumann:28} compares the players' security levels with each other and with the game's value.

\begin{theorem}[Minimax]
\label{Th:Minimax}
For every zero-sum, two-player game $u$,
\begin{enumerate}
	\item $\max_{\mu} \min_{\nu} U(\mu,\nu) = \min_{\nu} \max_{\mu} U(\mu,\nu)$; and
	\item $\val(u) = \max_{\mu} \min_{\nu} U(\mu,\nu)$.
\end{enumerate}
\end{theorem}

The Minimax Theorem implies that the informational asymmetry between $\max_\mu \min_\nu$ and $\min_\nu \max_\mu$ cannot be utilized by either player, that is, it does not negatively affect her expected utility if Eloise hands over to Abelard the strategy $\mu$ that maximizes $\min_\nu U(\mu,\nu)$ \emph{before Abelard makes his choice}. In the same vein, it does not negatively affect Abelard's expected utility if he hands over the strategy $\nu$ that minimizes $\max_\mu U(\mu,\nu)$ before Eloise makes her choice. 

It is easy to check that for any given mixed strategy $\mu$,
	\begin{equation}
		\min_\nu U(\mu,\nu) = \min_{0 \leq j < n} U(\mu,j),
	\label{Eq:MixedVersusPure}	
	\end{equation}
where $U(\mu,j)$ denotes the expected utility of Eloise if she plays $\mu$ against the pure strategy $j$. So, whenever Eloise hands over her strategy $\mu$, all Abelard needs to do is compute the expected utility $U(\mu,j)$ for each of his pure strategies $j$. If $\mu$ is an equilibrium strategy  and $j$ minimizes $U(\mu,j)$, $\val(u)$ is equal to $U(\mu,j)$.

Introduce the variable $\mu_i$ to represent the value $\mu(i)$ that Eloise's mixed strategy $\mu$ assigns to her pure strategy $i$. We can regard $\mu$ as the row vector
		\[
			\big[ \mu_0, \ldots, \mu_{m-1} \big].
		\]
Multiplying Eloise's strategy $\mu$ (as row vector) with $u$ yields the row vector
		\[
			\big[ U(\mu,0), \ldots, U(\mu,n-1) \big].
		\]
Reading Abelard's strategy $\nu$ as a column vector, $\mu u \nu$ is equal to $U(\mu,\nu)$.

We write $\row^u(i)$ for the $i$th row in $u$, which is a row vector, and $\col^u(j)$ for the $j$th column in $u$, which is a column vector. For a vector of values $v=[v_0,\ldots,v_{k-1}]$, let $\tally{v}$ denote the sum of its elements: $\sum_{0 \leq i < k} v_i$. Clearly, for our $u$, $\tally{\row^u(i)}$ coincides with the number of nonzero entries in the $i$th row in $u$. We say that $u$ is \emph{row balanced} if all its rows have the same sum: $\tally{\row^u(i)} = \tally{\row^u(i')}$, for all $0 \leq i, i' < m$. Similarly, we say that $u$ is \emph{column balanced} if $\tally{\col^u(j)} = \tally{\col^u(j')}$, for all $0 \leq j, j' < n$. A game is \emph{balanced} if it is both row and column balanced.

If Eloise plays $\mu$ and Abelard plays $j$, Eloise's expected utility $U(\mu,j)$ is the product of $\mu$ and $\col^u(j)$. Consequently, 
Eloise's task of maximizing 
	\[
		\min_{0 \leq j < n} U(\mu,j)
	\]
boils down to selecting a mixed strategy $\mu$ that maximizes the minimal element $v$ in 
	\[
		\big[ 
			\mu \col^u(0) ,\ldots, \mu \col^u(n-1)
		\big],
	\]	
that is, optimizing $v$ subject to the following constraints:
		\begin{align*}
			\mu \col^u(0) & \geq v \\[-1mm]
				& \vdots & \\
			\mu \col^u(n-1) & \geq v,
		\end{align*}
plus (for every $0 \leq i < n$):
		\[
			\mu_i \geq 0
		\]	
and
		\[
			\mu_0 + \ldots + \mu_{n-1} = 1.
		\]	
The latter $n+1$ constraints ensure that $\mu$ is a proper probability distribution. Modulo some rewriting, the above constraints constitute a linear programming problem. The solution, i.e., the optimized value for $v$, coincides with Eloise's security level in the underlying game, which coincides with its value by the Minimax Theorem.
	
As an example consider the game:
		\[
			\left[
				\begin{array}{cccc}
					1 & 1 & 0 & 0 \\
					0 & 1 & 1 & 0 \\
					1 & 0 & 1 & 0 \\
					1 & 1 & 1 & 0 \\
				\end{array} 
			\right],
		\]
which yields the following four constraints (in addition to the five constraints that ensure that $\mu$ is a probability distribution):
	\begin{align*}
		\mu_0 + \mu_2 + \mu_3 & \geq v \\
		\mu_0 + \mu_1 + \mu_3 & \geq v \\
		\mu_1 + \mu_2 + \mu_3 & \geq v \\
		0 & \geq v.
	\end{align*}
	Due to the fourth constraint, the maximum for $v$ is $0$ regardless of $\mu_0,\ldots,\mu_3$. Thus, whatever strategy Eloise plays, she has expected utility $0$, that is, Abelard has a winning strategy.
	
Flipping the bottom right value gives the game
		\[
			\left[\begin{array}{cccc}
			1 & 1 & 0 & 0 \\
			0 & 1 & 1 & 0 \\
			1 & 0 & 1 & 0 \\
			1 & 1 & 1 & 1 \\
			\end{array} \right],
		\]
yielding the same constraints as above, replacing the fourth by
	\[
		\mu_3 \geq v.
	\]
The maximum value for $v$ is $1$, realized by $\mu_3=1$ and $\mu_0=\mu_1=\mu_2=0$, reflecting the fact that the bottom strategy is winning for Eloise.
			
Finally, we consider an undetermined game:
		\[
			u = \left[\begin{array}{cccc}
			1 & 0 & 0 & 0 \\
			0 & 1 & 1 & 0 \\
			0 & 0 & 1 & 1 \\
			0 & 1 & 0 & 1 \\
			\end{array} \right]
		\]
yielding the following constraints:
	\begin{align*}
		\mu_0 & \geq  v \\
		\mu_1 + \mu_3 & \geq  v \\
		\mu_1 + \mu_2 & \geq  v \\
		\mu_2 + \mu_3 & \geq  v.
	\end{align*}
From the last three equations we derive that $\mu_1=\mu_2=\mu_3$.  So $\mu u$ is the row vector
	\[
		[ \mu_0, 2\mu_1, 2\mu_1, 2\mu_1 ].
	\]
From the first equation, it follows that $\mu_0 = 2\mu_1$ so the minimal element in this vector is only maximized by assignments for which $\mu_0 = 2\mu_1$. Since we require that $\mu$ be a probability distribution, there is only one such assignment: the one for which $\mu_0 = 2/5$ and $\mu_1=1/5$. Accordingly the value of the game is $2/5$.
	
It is tempting to replace the inequality symbols $\geq$ by the equality symbol $=$. Doing so does not affect the outcome of the latter game, but generally it is untrue that a maximizing $\mu$ yields a vector $\mu u$ of the form
	\[
		\big[ U(\mu,\tau_0), \ldots, U(\mu,\tau_{n-1}) \big]
	\]
of equal values. See for instance the game:
		\begin{equation}
		\label{Eq:DontReplaceGeqWithIdentity:Matrix}
			\left[\begin{array}{cccccc}
			1 & 0 & 0 & 0 & 0 & 1 \\
			1 & 0 & 0 & 1 & 0 & 0 \\
			0 & 1 & 1 & 1 & 0 & 0 \\
			1 & 1 & 0 & 0 & 1 & 0 \\
			0 & 0 & 1 & 0 & 1 & 0
			\end{array} \right]
		\end{equation}
Eloise's strategy $\mu$ such that $\mu_0 = \mu_1 = \mu_3 =  1/7$ and $\mu_2 = \mu_4 = 2/7$ maximizes $\min_\tau (\mu,\tau) = 3/7$. However, $U(\mu,\tau_2)=4/7$. To see that there is no maximizing $\mu$ that yields a vector $\mu u$ of equal values, consider the game's corresponding linear programming problem, replacing $\geq$ by $=$, we get
	\begin{align}
	\mu_0 + \mu_1 + \mu_3 & = v \label{Eq:DontReplaceGeqWithIdentity:Con1}\\
	\mu_2 + \mu_3 & = v \label{Eq:DontReplaceGeqWithIdentity:Con2}\\
	\mu_2 + \mu_4 & = v \label{Eq:DontReplaceGeqWithIdentity:Con3}\\
	\mu_1 + \mu_2 & = v \label{Eq:DontReplaceGeqWithIdentity:Con4}\\
	\mu_3 + \mu_4 & = v \label{Eq:DontReplaceGeqWithIdentity:Con5}\\
	\mu_0 + \mu_4 & = v \label{Eq:DontReplaceGeqWithIdentity:Con6}.
	\end{align}
Eqs.~(\ref{Eq:DontReplaceGeqWithIdentity:Con2}) and (\ref{Eq:DontReplaceGeqWithIdentity:Con3}) imply $\mu_3 = \mu_4$. In the same vein, Eqs.~(\ref{Eq:DontReplaceGeqWithIdentity:Con3}) and (\ref{Eq:DontReplaceGeqWithIdentity:Con4}) imply $\mu_1 = \mu_4$; Eqs.~(\ref{Eq:DontReplaceGeqWithIdentity:Con5}) and (\ref{Eq:DontReplaceGeqWithIdentity:Con6}) imply $\mu_0 = \mu_3$; and Eqs.~\ref{Eq:DontReplaceGeqWithIdentity:Con2} and \ref{Eq:DontReplaceGeqWithIdentity:Con5} imply $\mu_2 = \mu_4$. We conclude that $\mu_0 = \ldots = \mu_4 = 1/5$, contradicting Eqs.~(\ref{Eq:DontReplaceGeqWithIdentity:Con1}) and (\ref{Eq:DontReplaceGeqWithIdentity:Con2}).

\subsection{Bounds and characterizations}
For an $m \times n$ matrix $u$, we let $\colmin(u)$ denote
	\[
		\min\big\{ \tally{\col^u(0)}, \ldots, \tally{\col^u(n-1)} \big\},
	\]
and $\colargmin(u)$ the set of indices $0 \leq j < n$ for which
	\[
		\tally{\col^u(j)} = \colmin(u).
	\]
In a similar way we introduce $\rowmax$ and $\rowargmax$. We define 
	\[
		\floor(u) = \frac{\colmin(u)}{m}
	\]	
and 
	\[
		\ceil(u) = \frac{\rowmax(u)}{n}.
	\]
	

For instance, the matrix $u$ in (\ref{Eq:DontReplaceGeqWithIdentity:Matrix}) has $\floor(u) = 1/5$ and $\ceil(u) = 3/6$.

\begin{proposition}
\label{Prop:LowerBound1} 
For a game $u$,
	\begin{enumerate}
		\item $\floor(u) = \min_{\nu} U(\bar{\mu},\nu)$, where $\bar{\mu}$ is Eloise's uniform strategy; and
		\item $\floor(u) \leq \val(u)$.
	\end{enumerate}
\end{proposition}
	
\begin{proof}
Claim (1). Suppose $u$ is an $m \times n$ game. Eloise's strategy $\bar{\mu}$ assigns $1/m$ to each strategy $0 \leq i < m$. Multiplying $\bar{\mu}$ with $u$ yields:
	\[
		\big[ \tally{\col^u(0)}/m, \ldots, \tally{\col^u(n-1)}/m \big].
	\]
Abelard picks a strategy that yields 
	\[
		\min_j U(\bar{\mu},j)=\colmin(u)/m = \floor(u)
	\] 
for Eloise. By Eq.~(\ref{Eq:MixedVersusPure}), no mixed strategy of Abelard can outperform $j$, given that Eloise plays $\bar{\mu}$:
	\[
		\min_{j} U(\bar{\mu},j) = \min_{\nu} U(\bar{\mu},\nu).
	\]
	
Claim (2). Playing $\bar{\mu}$ yields at least $\floor(u)$ for Eloise, by Claim (1). So Eloise can secure at least $\floor(u)$ in $u$.
\end{proof}
	
\begin{proposition}
\label{Prop:UpperBound1}
For a game $u$,
	\begin{enumerate}
		\item $\ceil(u) = \max_{\mu} U(\mu,\bar{\nu})$, where $\bar{\nu}$ is Abelard's uniform strategy; and
		\item $\val(u) \leq \ceil(u)$.
	\end{enumerate}
\end{proposition}
	
\begin{proof}
Analogous to the proof of Proposition \ref{Prop:LowerBound1}.
\end{proof}	
	
\begin{proposition} 
\label{Prop:Equivalence:1}
For a balanced game $u$,
	\begin{enumerate}
		\item $\val(u)=\floor(u)=\ceil(u)$; and
		\item the strategy profile $(\bar{\mu},\bar{\nu})$ is an equilibrium in $u$.
	\end{enumerate}
\end{proposition}
	
\begin{proof}
Claim (1). Suppose $u$ is an $m\times n$ game. By  the fact that the game is column balanced it follows from Proposition \ref{Prop:LowerBound1} that the value of the game is at least 
	\[
		\floor(u)=\frac{\tally{\col^u(0)}}{m}.
	\]
	By  the fact that the game is row balanced it follows from Proposition \ref{Prop:UpperBound1} that the value of the game is at most
	\[
		\ceil(u)=\frac{\tally{\row^u(0)}}{n}.
	\]
	Since $u$ is row balanced, it has precisely $m\tally{\row^u(0)}$ entries with a $1$; since it is column balanced, it has precisely $n\tally{\col^u(0)}$ entries with a $1$. Hence, $m\tally{\row^u(0)} = n\tally{\col^u(0)}$, and it follows that the upper and lower bounds coincide, since we have that
	\[
		\frac{\tally{\row^u(0)}}{n} = \frac{\tally{\col^u(0)}}{m}.
	\]
The equality $\val(u)=\ceil(u)$ can be derived similarly.

Claim (2). Observe that:
	\[
		U(\bar{\mu},\bar{\nu}) 
		\leq 
		\max_{\mu} U(\mu,\bar{\nu}) 
		= 
		\min_\nu U(\bar{\mu},\nu) 
		\leq 
		U(\bar{\mu},\bar{\nu}).
	\]
The equality follows from Claim (1) and Propositions \ref{Prop:LowerBound1}.1 and \ref{Prop:UpperBound1}.1; the inequalities follow from the definition of $\max$ and $\min$, respectively. It follows that
	\[
		\max_{\mu} U(\mu,\bar{\nu}) 
		=
		U(\bar{\mu},\bar{\nu})
		= 
		\min_\nu U(\bar{\mu},\nu).
	\]

From this equality it follows that for every $\mu$ and $\nu$,
	\[
		U(\mu,\bar{\nu})
		\leq
		\max_{\mu} U(\mu,\bar{\nu}) 
		=
		U(\bar{\mu},\bar{\nu})
		= 
		\min_\nu U(\bar{\mu},\nu)
		\leq
		U(\bar{\mu},\nu),
	\]
and $(\bar{\mu},\bar{\nu})$ is an equilibrium of $u$, by definition.
\end{proof}

\begin{example}
\label{Ex:Prop:Equivalence:1}
Consider the $n \times n$ matrix $u$ with 1s on the diagonal:
\[
			\left[
			\begin{array}{cccccc}
			1 & 0 & \cdots & 0 & 0 \\
			0 & 1 &  & 0 & 0 \\[2mm]
			\vdots &  & \ddots &  & \vdots \\[3mm]
			0 & 0 &  & 1 & 0 \\
			0 & 0 & \cdots & 0 & 1
			\end{array} 
			\right].
\]
This matrix is obviously balanced. Whence, by Proposition \ref{Prop:Equivalence:1}, the value is $1/n$. Note that $u$ is isomorphic to games $\Gamma(\str{M},\phimp)$ of the Matching Pennies sentence on structures $\str{M}$ with $n$ elements.
\end{example}

A \emph{row submatrix} $u'$ of $u$ is any matrix that can be obtained by deleting any number of rows from $u$ (in any order). 
	
\begin{proposition}
\label{Prop:LowerBound2} 
For a game $u$,
	\[
		\max_{u'} \floor(u') \leq \val(u),
	\]
where $u'$ ranges over the nonempty row submatrices of $u$.
\end{proposition}
	
\begin{proof}
Suppose $u$ is an $m \times n$ matrix. For a row submatrix $u'$ of $u$ with $m'$ rows, define the mixed strategy $\mu$ in the game of $u$ for which
	\[
		\mu(i)	 = \left\{ 
		\begin{array}{ll}
			1/m' 	& \textrm{if $\row^u(i)$ is also a row in $u'$} \\
			0 		& \textrm{otherwise.} \\
		\end{array} \right.
	\]
	The strategy $\mu$ plays all rows in $u'$ with equal probability, and does not play any of the rows in $u$ that do not sit in $u'$. Hence $\mu u$ yields
	\[
		\big[ \tally{\col^{u'}(0)}/m', \ldots, \tally{\col^{u'}(n-1)}/m' \big].
	\]
The minimal value in this vector equals $\floor(u')$. So Eloise can iterate through all row submatrices $u'$. Then, she can secure $\floor(u')$ by playing the mixed strategy associated with a $u'$ that maximizes $\floor(u')$.
\end{proof}	

\begin{proposition}
\label{Prop:Equivalence:2} 
Let $u'$ be a row submatrix of the game $u$. If $u'$ is balanced and
	\[
		\rowmax(u) = \rowmax(u'),
	\]
then 
	\begin{enumerate}
		\item $\val(u)=\val(u')$; and
		\item the strategy profile $(\bar{\mu},\bar{\nu})$ of Eloise's and Abelard's uniform strategies in $u'$ respectively, is an equilibrium in $u$.
	\end{enumerate}
\end{proposition}

\begin{proof}
Claim (1). By Propositions \ref{Prop:LowerBound2} and \ref{Prop:Equivalence:1}.1,
\[
	\val(u') \leq \val(u),
\]
and by Propositions \ref{Prop:UpperBound1}.2,
\[
	\val(u) \leq \ceil(u).
\]
Since $u'$ is balanced, it follows from Proposition \ref{Prop:Equivalence:1}.1 that 
\[
	\floor(u') = \val(u') = \ceil(u').
\]
Since $\rowmax(u) = \rowmax(u')$, we have that $\ceil(u) = \ceil(u')$. Hence
\[
	\val(u') = \val(u).
\]

Claim (2). By Proposition \ref{Prop:Equivalence:1}.2, the pair of uniform strategies $(\bar{\mu},\bar{\nu})$ is an equilibrium in $u'$, whence for every mixed strategy $\nu$ in $u'$,
	\[
		U'(\bar{\mu},\bar{\nu}) \leq U'(\bar{\mu},\nu),
	\]
where $U'$ is the expected utility function of $u'$. Since $u'$ is a submatrix of $u$, $U'$ and $U$ agree on every pair of mixed strategies in the domain of $U'$. Since $u'$ is a \emph{row} submatrix of $u$, Abelard's set of strategies in the two games coincide. Therefore, the latter inequality boils down to
	\[
		U(\bar{\mu},\bar{\nu}) \leq U(\bar{\mu},\nu),
	\]
for every mixed strategy $\nu$ of Abelard in $u$.	Remark that strictly speaking, in this inequality, $\bar{\mu}$ denotes Eloise's mixed strategy in $u$ that is uniform in the strategies that are shared between $u$ and $u'$.

If Abelard plays $\bar{\nu}$, Eloise can secure the maximal element in $u \bar{\nu}$. For the transposition of this vector write
	\[
		\left[ 
			w_0, \ldots, w_{m'-1} , w_{m'} , \ldots, w_{m-1}
		\right],
	\]
assuming $u$ is an $m \times n$ game and $u'$ is an $m' \times n$ game, with $m' \leq m$. Since $u'$ is balanced all its rows have the same sum, namely  $\rowmax(u')=\rowmax(u)$. Thus,
	\[
		\rowmax(u)/n = w_0 = \ldots = w_{m'-1}.
	\]
Here we assume that the first $m'$ rows in $u$ constitute $u'$, which obviously goes without loss of generality. Furthermore, we may assume that
	\[
	w_{m'-1} \geq w_{m'} \geq \ldots \geq w_{m-1}.
	\]
Since $\bar{\mu}$ only assigns non-zero probabilities to the first $m'$ strategies, no mixed strategy $\mu$ of Eloise in $u$ can outperform $\bar{\mu}$, given that Abelard plays $\bar{\nu}$:
	\[
		U(\mu, \bar{\nu}) \leq U(\bar{\mu},\bar{\nu}).
	\]
It follows that $(\bar{\mu},\bar{\nu})$ is an equilibrium in $u$.
\end{proof}

\section{Applications}
\label{Sec:Applications}

We give two applications of the tools developed to equilibrium semantics.

\subsection{Birthday problem}
\label{Ssec:BirthdayProblem}

Considered is a party attended by $m$ persons. What is the probability that there is a pair of individuals that have the same birthday? A straightforward combinatorial argument shows that the probability exceeds 50 per cent when $m > 20$.

The ``birthday problem'' can be redefined in terms of drawing $m$ balls (number of guests) from an urn of $n$ balls (number of birthdays) with replacement. We are interested in the odds that we draw the same ball twice.

Given the first ball $b_0$, the probability that the second ball $b_1$ is not equal to $b_0$ is 
	\[
		\frac{n-1}{n}.
	\]
Similarly, given $i$ distinct balls $b_0,\ldots,b_{i-1}$, the probability that the next ball is not among the balls drawn earlier is 
	\[
		\left( \frac{n}{n} \right)
		\left( \frac{n-1}{n} \right)
		\left( \frac{n-2}{n} \right) 
		\ldots 
		\left( \frac{n-i}{n} \right).
	\]
It follows that the odds that $b_0,\ldots,b_{m-1}$ are all distinct is
	\begin{equation}
	\label{Eq:BirthdayEquation}
		\frac{n!}{n^m(n-m)!}.
	\end{equation}

We define the process of randomly drawing $m$ balls from the urn in IF logic. Consider the IF sentence
	\[
		\phi_m = \forall x_0 \ldots (\forall x_{m-1}/X_{m-1}) (\exists x_m / X_m) \ldots (\exists x_{2m-1}/X_{2m-1}) \psi_m
	\]
where $X_k = \{x_0,\ldots, x_{k-1}\}$ and 
	\[
 		\psi_m = \bigvee_{0 \leq i < m} \ \bigvee_{i < j < m} (x_i + x_{i+m}) = (x_j + x_{j+m}),
 	\]
in which the addition operator is defined as $a_k + a_l = a_{k + l mod n}$ assuming that the $n$ objects in the domain at hand are labeled $a_1,\ldots,a_{n-1}$. In an extensive game of $\phi_m$, Abelard and Eloise pick $m$ objects each. Abelard's first object $a_0$ is added to Eloise's first object $a_m$, and so on for the other $m-1$ objects. Since Eloise does not know any of Abelard's choices, the object $b_i = a_i + a_{m+i}$ is effectively chosen at random. Eloise wins if there is a pair of sums $b_i=b_j$ with $i<j$; otherwise Abelard wins. 

\begin{proposition}
\label{Prop:BirthdayProblem}
The probability that a random sample of $m$ elements from a set of $n$ elements (with replacement) contains at least one pair of duplicates is $\val(\str{M},\phi_m)$ for any structure $\str{M}$ of size $n$ in which the addition operator is defined as above.
\end{proposition}

\begin{proof}
Both Abelard and Eloise have $n^m$ choices for their respective quantifiers. Then, Eloise has two disjunction choices knowing Abelard's and her own moves (or more if we permit ourselves only binary disjunctions). Given this knowledge, she has an ``optimal substrategy'':  chose the first disjunct that holds with respect to the chosen objects, if such a disjunct exists; otherwise, select an arbitrary disjunct. It is clear that this strategy outperforms or is equivalent to any other substrategy she may have, given the objects selected for the objects. It has been shown that we can remove such ``weakly dominated'' and ``payoff equivalent'' strategies from the strategic game, without affecting the game's value \cite[Proposition 7.25]{MannEtAl:11}. 

Thus we can focus on the game that is the result of eliminating weakly dominated and payoff equivalent strategies. In this reduced game, each of Abelard's and Eloise's strategies corresponds to an ordered series of $n$ objects from $M$. So, each player has $n^m$ strategies. Consider any strategy $(a_0,\ldots,a_{m-1})$ of Abelard. Let us count the number of strategies $(a_m,\ldots,a_{2m-1})$ of Eloise against which Abelard's strategy results in a loss for Eloise. The object $a_m$ can be chosen in any way we want, yielding $n$ choices. Write $b_i$ for $a_i + a_{m + i}$. Given a series of $i$ distinct objects $b_0,\ldots,b_{i-1}$, there are $n-i$ objects $a_{m+i}$ for which 
\[
	a_i + a_{m+i} \notin \{ b_0 \ldots, b_{i-1} \}.
\]
So every strategy of Abelard loses against $n(n-1)(n-2)\ldots(n-m) = n!/(n-m)!$ strategies of Eloise, and wins against
\begin{equation}
\label{Eq:Birthday:Floor}
	n^m - \frac{n!}{(n-m)!}
\end{equation}
strategies. The strategy of Abelard was chosen without loss of generality; it follows that $u$ is row balanced and that the value (\ref{Eq:Birthday:Floor}) equals $n^m - \colmin(u)$, where $u$ is the matrix of the reduced strategic game of $\phi_m$ on $\str{M}$. Since $u$ has $n^m$ columns, 
\[
	\floor(u)
	=
	\frac{\colmin(u)}{n^m}
	=
	\frac{n!}{n^m(n-m)!}
	=
	(\ref{Eq:BirthdayEquation}).
\]

A symmetric argument shows that $u$ is also row balanced. By Proposition \ref{Prop:Equivalence:1} it follows that the value of $\phi_m$ on $\str{M}$ is (\ref{Eq:BirthdayEquation}).
\end{proof}

What is at stake in the birthday problem is the number of \emph{pairs} of party goers $m(m-m)/2$, which is quadratic in $m$. This is reflected in the IF sentence $\phi_m$, which has $2m$ quantifiers to simulate the random selection of $m$ objects. The quantifier-free prefix $\psi_m$ has $m(m-1)/2$ disjuncts, one for each pair of distinct party goers.

\subsection{Universal hashing}
\label{Ssec:UniversalHashing}

In this section we use IF logic to describe the game theory behind \emph{universal hashing}, a notion from computer science. Hashing functions are used to map an unknown set $S$ from a set of objects $U$ called \emph{keys}, to a set of \emph{hash values} $V$. If we have a linear order on $V$, then we can use a hash function to store the elements from $S$ in a hash table, which allows for binary search.

For instance, we can think of $U$ as the collection of all finite strings with at most 100 characters, $S$ as the set of Dutch names, and $V$ as some range of integers in an administration system. Surely, we do not want to reserve as many integers as there are elements in $U$; a hash function helps us transfer an arbitrary key from $U$ to a hash value in $V$.

By the pigeon hole principle, if $S$ has more elements than $V$, for every hash function, there is a pair of keys $k,l \in U$ that are mapped to the same hash value. Such a pair of objects is said to \emph{collide}. Collision handling is typically resource intensive, for which reason we want to select a hash function that minimizes the expected number of collisions not knowing the set $S$ of keys that will actually materialize. The following fragment explains hashing as a game.

\begin{quotation}
``If a malicious adversary chooses the keys to be hashed by some fixed hash function, then he can choose $n$ keys that all hash to the same slot, yielding an average retrieval time [that is linear in $n$]. Any fixed hash function is vulnerable to such terrible worst-case behavior; the only effective way to improve the situation is to choose the hash function \emph{randomly} in a way that is \emph{independent} of the keys that are actually going to be stored. This approach, called \emph{universal hashing}, can yield provably good performance on average, no matter what keys are chosen by the adversary.

The main idea behind universal hashing is to select the hash function at random from a carefully designed class of functions at the beginning of execution. [\ldots] Poor performance occurs only when the compiler chooses a random hash function that causes the set of identifiers to hash poorly, but the probability of this situation occurring is small and is the same for any set of identifiers of the same size.'' \cite[pp.\ 232--3]{CormenEtAl:03}
\end{quotation}

We will work on \emph{hash structures} 
	\[
		\str{M}=\big(M,U^\str{M},(f^\str{M}_i)_{0 \leq i < n}\big),
	\] 
where $U^\str{M} \subseteq M$ are the keys, $M - U^\str{M}$ the hash values, $(f^\str{M}_i)_{i}$ is the series of all functions from $U^\str{M}$ to $M-U^\str{M}$ and $i$ is an injective indexing of thereof. The assumption that a hash structure records all possible hash functions with respect to the set of keys and values at hand is very strong. We will return to this assumption before we leave this section.

The game dynamics described in the first part of the first paragraph are captured by:
\[
	\phi_\mathrm{H} 
	= 
	\bigvee_i \forall x \forall y \Big[ \big(U(x) \land U(y) \land x \neq y \big) \to f_i(x) \neq f_i(y) \Big],
\]
in which the operator $\bigvee_i$ is object language.

In a game of $\phi_\mathrm{H}$ on a hash structure $\str{M}$, triggered by $\bigvee_i$ Eloise chooses the index $i$ of the hash function $f_i^\str{M}$. Then, Abelard in the capacity of malicious adversary chooses two keys for $x$ and $y$. If they collide with respect to $f_i^\str{M}$, Abelard wins. As we pointed out above, by the pigeonhole principle, Abelard has a winning strategy whenever $|U| > |V|$.

As explained in the remainder of the quotation, in the universal hashing scenario, Eloise tries to confuse Abelard by drawing her hash function at random. We will show that the optimal way to randomly select a hash function coincides with Eloise's equilibrium strategy in the game described by
\[
	\phi_\mathrm{UH} = \bigvee_i (\forall x/i) (\forall y/i) \Big[ \big(U(x) \land U(y) \land x \neq y \big) \to f_i(x) \neq f_i(y) \Big].
\]

Assume $U=\{k_0,\ldots,k_{n-1}\}$ and $V = \{0,\ldots,m-1\}$. For a function $f:U \to V$, let $\mathcal{P}_f$ be the set of its pre-images:
	\[
		\mathcal{P}_f = \{ f^{-1}(v) : v \in V \}.
	\] 

The \emph{degree} of a function $f$ is the difference between the sizes of the largest and smallest pre-image in $\mathcal{P}_f$:
	\[
		\max\big\{ |P| : P \in \mathcal{P}_f \big\} - \min\big\{ |P| : P \in \mathcal{P}_f \big\}.
	\]
For every $U$ and $V$ there is a function $f:U\to V$ of degree (at most) $1$, see for instance the function
	\[
		f(k_i) = i \bmod{m}.
	\]
This function has in fact degree $0$ whenever $n \bmod{m} = 0$. 

In the context of the game $\Gamma(\str{M},\phi_{\mathrm{UH}})$, $S_d$ denotes the set of strategies of Eloise that pick indices $i$ for which $f_i^\str{M}$ has degree $d$. Each strategy of Abelard corresponds to the pair of keys $(k,l)$ it assigns to $x$ and $y$. We write $T^*$ for the set of strategies of Abelard that assign distinct keys to $x$ and $y$.

\begin{proposition}
Let $d = \min\{1,n \bmod{m}\}$ and $S^* = S_d$. The pair of uniform strategies $(\bar{\mu},\bar{\nu})$ over $S^*$ and $T^*$ respectively is an equilibrium in $\Gamma=\Gamma(\str{M},\phi_{\mathrm{UH}})$.
\end{proposition}

\begin{proof}
Every strategy of Abelard that does not sit in $T^*$ is losing for Abelard, and therefore weakly dominated by every strategy in $T^*$. Whence, we may discard these dominated strategies from our analysis, see \cite[Chapter 7]{MannEtAl:11}. $T^*$ contains $n^2 - n$ strategies.

We consider a function $\lambda$ on the set of functions of type $U \to V$. Let $f$ be one particular function of this type. Write $i_{\max}^f$ for a value with largest pre-image: 
\[
	i_{\max}^f = \argmax_i \big\{ |f(i)^{-1}| \big\}.
\]
Let $k^*$ be any element in $f(i_{\max}^f)^{-1}$. Similarly, write $i_{\min}^f$ for a value with smallest pre-image: 
\[
	i_{\min}^f = \argmin_i \big\{ |f(i)^{-1}| \big\}.
\]
Then, $\lambda$ sends $f$ to the function $\lambda(f)$ for which
	\[
		\lambda(f)(k) 
		= 
		\left\{ 
			\begin{array}{ll}
				i_{\min}^f & \textrm{if $k = k^*$} \\
				f(k) & \textrm{otherwise.}
			\end{array}
		\right.
	\]
If a collision occurs, it is more likely to happen between keys that are sent to $i_{\max}^f$ than to keys sent to $i_{\min}^f$. The operator $\lambda$ levels the probability that a collision occurs at $i_{\max}^f$ and the probability that one occurs at $i_{\min}^f$. We will see that it also decreases the likelihood of a collision appearing in the first place.

Write $\sigma$ and $\sigma'$ for Eloise's strategies in $\Gamma$ that pick the indices of $f$ and $f'=\lambda(f)$, respectively. We show that $\sigma$ loses against more strategies of Abelard than $\sigma'$. Eloise's strategy $\sigma$ loses against Abelard's strategies in 
	\[
		L_\sigma = \bigcup_{P \in \mathcal{P}_f} L_P.
	\]
where
	\[
		L_P = \big\{ (k,l) \in T^* : k,l \in P \big\}.
	\]
Similarly, $\sigma'$ loses against the strategies in 
	\(
		L_{\sigma'} = \bigcup_{P \in \mathcal{P}_{f'}} L_P
	\).
To show that $|L_\sigma| > |L_{\sigma'}|$, it suffices to show that 
	\begin{equation}
	\label{Eq:UniversalHashing:100}
		\big|L_{f(i_{\max}^f)^{-1}}\big| + \big|L_{f(i_{\min}^f)^{-1}}\big|
		> 
		\big|L_{f'(i_{\max}^f)^{-1}}\big| + \big|L_{f'(i_{\min}^f)^{-1}}\big|,
	\end{equation}
because the other pre-images are shared between $f$ and $f'$. For a pre-image $P$ of size $z$, $L_{P}$ contains $z(z-1)$ elements. Whence, if $f(i_{\min}^f)^{-1}$ contains $x$ elements, 
	\begin{align*}
		\big|L_{f(i_{\max}^f)^{-1}}\big| & = x(x-1) \\
		\big|L_{f'(i_{\max}^f)^{-1}}\big| & = (x-1)(x-2).
	\end{align*}
Likewise, if $f(i_{\min}^f)^{-1}$ contains $y$ elements, 
	\begin{align*}
		\big|L_{f(i_{\min}^f)^{-1}}\big| & = y(y-1) \\
		\big|L_{f'(i_{\min}^f)^{-1}}\big| & = (y+1)y.
	\end{align*}
Hence, Eq.~(\ref{Eq:UniversalHashing:100}) reduces to 
	\[
		x(x-1) + y(y-1) > (x-1)(x-2) + (y+1)y,
	\]
which is the case if $x > y + 1$.

We leave it as an exercise to the reader to verify that for any function $f$ of degree $d'>1$, there is a finite series
	\[
		f,\lambda(f),\lambda\big(\lambda(f)\big),\ldots, \lambda\big(\ldots\lambda\big(\lambda(f)\big)\ldots\big)
	\]
of which the final element has degree $d'-1$. Thus, iterative application of $\lambda$ ultimately yields a function with degree $d$. As we have just shown above, every application of $\lambda$ yields a hash function that further reduces the number of strategies against which the strategy loses that picks the index of that hash function. If we reach a function with degree $d$, applying $\lambda$ no longer reduces this number. It can further be checked that as long as $\lambda$ can be applied, $x > y + 1$.

It is easy to see that every two functions with degree $d$ suffer from an equal number of collisions. Let $\Gamma^*$ be the subgame of $\Gamma$ induced by $S^*$ and $T^*$, that is, $\Gamma^*$ is of the form $(S^*,T^*,u^*)$ where $u^*(\sigma,\tau)=u(\sigma,\tau)$, for any $\sigma \in S^*$ and $\tau \in T^*$.

Our $\lambda$-argument showed that $\rowmax(u^*)=\rowmax(u)$. Since every two functions of degree $d$ suffer from the same number of collisions, $u^*$ is row balanced. It is easy to see that $u^*$ is also column balanced. We apply Proposition \ref{Prop:Equivalence:2}.2 to infer that $(\bar{\mu},\bar{\nu})$ is an equilibrium in $\Gamma$.
\end{proof}

We have seen in the section on the birthday problem, Section \ref{Ssec:BirthdayProblem}, that we can simulate drawing random objects, but we were incapable of extending this method to drawing random functions. It appears to us that if we have a means to express randomization over functions, we can utilize this mechanism to express universal hashing without assuming structures that carry all possible hash structures. 

If it turns out that IF logic can express random functions, or if it turns out that it cannot, we may be able to use this to prove new upper or lower bounds on the expressive power of IF under equilibrium semantics. In fact, if the former is the case --- i.e., IF logic can express random functions --- it would be most interesting to see if it has natural fragments that coincide with randomized complexity classes in the style of Fagin's Theorem, which showed that second-order existential logic coincides with the complexity class NP \cite{Fagin:74}.
 
\section{Acknowledgements}

We gratefully acknowledge Fausto Barbero for careful proofreading.

\end{document}